\documentclass[12pt, reqno]{amsart}
\usepackage{amscd} 
\usepackage{amsfonts} 
\usepackage{amssymb} 
\usepackage{latexsym} 
\usepackage[arrow, matrix, curve]{xy}

\newcommand{\ncm}{\newcommand}

\newtheorem{theorem}{Theorem}[section]
\newtheorem{prop}[theorem]{Proposition}
\newtheorem{lemma}[theorem]{Lemma}
\newtheorem{cor}[theorem]{Corollary}
\newtheorem{lem&def}[theorem]{Lemma \& Definition}
\newtheorem{definition}[theorem]{Definition}
\newtheorem{example}[theorem]{Example}

\def\M{\mathcal{M}}

\def\C{\mathbb{C}\,}

\def\N{\mathbb{N}\,}

\ncm{\End}{\mbox{\rm End}\,}

\def\Hom{\mbox{\rm Hom}}

\def\Ind{\mbox{\rm Ind}}
\def\Res{\mbox{\rm Res}}
\def\|{\, | \,}

\def\id{\mbox{\rm id}}

\def\into{\hookrightarrow}

\def\bra{\langle}
\def\ket{\rangle}

\ncm{\rarr}[1]{\stackrel{#1}{\longrightarrow}}
\ncm{\larr}[1]{\stackrel{#1}{\longleftarrow}}

\def\-2{_{(-2)}}
\def\-1{_{(-1)}}
\def\0{_{(0)}}
\def\1{_{(1)}}
\def\2{_{(2)}}
\def\3{_{(3)}}

\def\du1{\hat 1}

\begin{document}
\title[Normality characterized by a  tower condition]{A tower condition \\ characterizing normality}
\author[L.~Kadison]{Lars Kadison} 
\address{Departamento de Matematica \\ Faculdade de Ci\^encias da Universidade do Porto \\ 
Rua Campo Alegre 687 \\ 4169-007 Porto} 
\email{lkadison@fc.up.pt} 
\thanks{}
\date{} 

\begin{abstract}
We define left relative H-separable tower of rings and continue a study of these begun by Sugano.  
  It is proven that a progenerator extension  has right depth $2$ if and only if the ring extension together with its right  endomorphism ring is a left relative H-separable tower.  In particular, this applies
to twisted or ordinary Frobenius extensions with surjective Frobenius homomorphism.  For example, normality for Hopf subalgebras of finite-dimensional Hopf algebras is also characterized in terms of this tower condition.
\end{abstract} 
\maketitle

\section{Introduction and Preliminaries}
\label{one}

Depth two is a bimodule condition on subrings that is equivalent to usual notions of normality for subgroups \cite{BDK}, Hopf subalgebras \cite{BK2} and semisimple complex subalgebra pairs \cite{BKK}; 
in addition, depth two is a condition of normality for a ring extension that makes it  a Galois extension with respect to a right bialgebroid coaction \cite{KS, LK2008}.
The right depth two condition on an algebra extension $A \supseteq B$ is that the natural $A$-$B$-bimodule $A \otimes_B A$ is isomorphic to a direct summand of a  natural
$A$-$B$-bimodule $A \oplus \cdots \oplus A$: in symbols this is  ${}_AA \otimes_B A_B \oplus * \cong {}_AA^n_B$. If  $A$ is a finite-dimensional Hopf algebra and $B$ is a Hopf subalgebra of $A$, it is shown in \cite{BK2} that $A \supseteq B$ has right (or left) depth $2$ if and only if $B$ is a normal Hopf subalgebra of $A$ (i.e., $B$ is invariant under either the left or right adjoint actions).  If $A$ is a finite-dimensional group algebra $\C G$, its module theory is determined by the character theory of $G$, and the right depth two condition on a group subalgebra $B = \C H$ in $A$ is determined by a matrix inequality condition on the induction-restriction table for the irreducible characters of $G$ and subgroup $H$; for more on this, depth greater than $2$ as well as modular representations, see \cite{BuK, BKK, BDK, D, F, FKR}. 

For the reasons just given  we call a subring $B \subseteq A$ satisfying the right depth
$2$ condition above a right normal subring (and the ring extension $A \supseteq B$
a right normal extension), the details appearing in Definition~\ref{def-normal}.  We find a characterization of normal Frobenius extensions with surjective Frobenius homomorphism (such as finite Hopf-Galois extensions with surjective trace map \cite{KN}), in terms of an old one-sided
H-separability condition on a tower of rings $A \supseteq B \supseteq C$ appearing in Sugano's \cite[Lemma 1.2]{S}.  This condition is interesting and we gather into ``Sugano's Theorem'' (Theorem~\ref{th-Sug}) the results for a tower satisfying this condition in \cite{S}, providing a different proof.  We show in Section~\ref{three} that a ring extension $A \supseteq B$
with the natural module $A_B$ a progenerator,  together with its right endomorphism ring $\End A_B$, forms a tower satisfying Sugano's condition, called ``left relative H-separable,''  if and only if $A \supseteq B$ is a right normal extension. 
In Corollary~\ref{cor-Frob} it is noted that  Theorem~\ref{th-charD2Hsep} establishes the same tower  characterization of normality for  
$\beta$-Frobenius extensions with surjective Frobenius homomorphism.  For example, an arbitrary Hopf subalgebra of a finite-dimensional Hopf algebra is such a twisted Frobenius extension: then  Corollary~\ref{cor-Hopfish} characterizes a normal Hopf subalgebra in terms of its right endomorphism algebra.   A new proof that right normality is equivalent to left normality for Frobenius extensions with the surjectivity condition is noted in Corollary~\ref{cor-leftright}.  

\subsection{H-separable extensions} A ring extension $A \supseteq B$ is H-separable if $A \otimes_B A \oplus * \cong A^n$ as natural $A$-bimodules \cite{Hir}. The notion of H-separability extends certain nice results for Azumaya algebras to ring extensions.  For example, the Azumaya isomorphism of the enveloping algebra and the endomorphism algebra is extended for an H-separable ring extension $A \supseteq B$ to a bimodule isomorphism, $A \otimes_B A \cong \Hom (R_Z, A_Z)$ where $R$ is the centralizer of $B$ in $A$ and $Z$ is the center of $A$ \cite{H}.  One also shows that $A \supseteq B$ is a separable extension, and if this is additionally split, that
$R$ is a separable algebra over $Z$ \cite{H,S, NEFE}.  Any bimodule $M$ over $A$ has a generalized Azumaya isomorphism $M^A \otimes_Z R \cong M^B$ between the $A$- and $B$-centralized elements of $M$.  

H-separable extension theory was one of the motivational models for  \cite{KS} which extends to ring  theory the notion of  depth $2$ for free Frobenius extensions in \cite{KN} (see \cite[Examples 3.6, 4.8, 5.8]{KS}, another toy model being Lu's Hopf algebroids on an algebra).  Examples of H-separable extension come from tensoring Azumaya algebras with other algebras, or looking at certain subalgebras within Azumaya algebras; certainly group algebra and Hopf algebra extensions are trivial if H-separable, which is true generally \cite{LK2013} but easier to prove with characters (see Proposition~\ref{prop-char} in this paper). 
The more general notion of depth two ring extension welcomes  examples  from 
 Hopf-Galois extensions including normal Hopf subalgebras; indeed, depth two ring extension is equivalent, with one other condition (balanced module), to a Galois extension with bialgebroid coactions \cite{KS, LK2008}, where bialgebroid is the good generalization of bialgebra from algebras to ring extensions.

 An H-separable extension $A \supseteq B$  is the H-depth $n = 1$ case of odd minimal H-depth  $d_H(B,A) = 2n-1$ where $A^{\otimes_B n} \sim A^{\otimes_B (n+1)}$ as $A$-bimodules (\cite{LK2011}, see below in this section for H-equivalent modules).    We show in Propositon~\ref{prop-equalityofHdepth} that a relative separable and H-separable tower $A \supseteq B \supseteq C$ has equality of minimal H-depth, $d_H(B,A) = d_H(C,A)$.  We note a different proof of Sugano's theorem~\ref{th-Sug} (cf. \cite{S}) that shows that in such a relative H-separable tower  there is a close relation between depth $1$
(centrally projective)
and H-depth $1$ (H-separable) extensions $B \supseteq C$ and $A^C \supseteq A^B$, as well as split and separable extensions.  

The following unpublished characterization of H-separable extensions is useful below.  

\begin{prop}
\label{folklore?}
Let $A \| B$ be a ring extension.  Then $A \| B$ is H-separable if and only if 
for each $A$-module $N$, its restriction and induction satisfies $\Ind^A_B \Res^A_B N \oplus * \cong N^m$ for some
$m \in \N$ via two natural transformations.  Consequently, if
$A \| B$ is H-separable and $A$-modules $V,W$ satisfy $V_B \oplus * \cong W_B$, then $V_A \oplus * \cong W_A^m$ for some $m \in \N$.  
\end{prop}
\begin{proof}
The second statement follows from the fact that $A \| B$ is also separable, so that
$V \otimes_B A \rightarrow V$, $v \otimes a \mapsto va$ is a natural split epi.
Note that $V \otimes_B A \oplus * \cong W \otimes_B A$, so that
the second statement follows from the characterization in the first statement.  

($\Rightarrow$)  Since $A \otimes_B A \oplus * \cong A^n$ as $A$-bimodules, the
implication follows from tensoring this by $N \otimes_A -$.  Naturality follows from
looking more carefully at the mappings, starting with a module homomorphism
$g: N_A \rightarrow {N'}_A$.  Another characterization of H-separability is that
there are elements $e_i \in (A \otimes_B A)^A$ and $r_i \in A^B$ ($i = 1,\ldots,n$) such that $1 \otimes_B 1 = \sum_i r_i e_i$.  For each module $M_A$, define natural transformations
$\tau_M: M \otimes_B A \rightarrow M_A^n$ by $\tau_M(m \otimes_B a) = 
(mr_1 a,\ldots,mr_n a)$, and $\sigma_M: {M_A}^n \rightarrow M \otimes_B A$
by $\sigma_M(m_1,\ldots,m_n) = \sum_i m_i e_i$; note that $\sigma_M \tau_M = \id_{M\otimes_B A}$  and the naturality commutative square follows readily.  

($\Leftarrow$) Let $N = A$ in the hypothesis using natural transformations (as above) $\sigma_N$ and $\tau_N$.  Then there are $A$-bimodule homomorphisms
(from naturality) $\tau_A: A \otimes_B A \rightarrow A^m$ and $\sigma_A: A^m \rightarrow A \otimes_B A$ such that $\sigma_A(\tau_A(1 \otimes_B 1))) = 1 \otimes_B 1$.  Then $\tau_A(1 \otimes_B 1) = (r_1,\ldots,r_m) \in (A^B)^m$
and, denoting the canonical basis of $A^m$ by $\{ c_1,\ldots,c_m \}$,
$\sigma_A(c_i) = e_i \in (A \otimes_B A)^A$.  Since $\tau_A$ is a section of $\sigma_A$, the equation $1 \otimes_B 1 = \sum_{i=1}^m r_ie_i$ follows; thus $A \| B$ is H-separable.    
\end{proof}
\subsection{Left relative separable ring towers} In this paper a \textit{tower of rings} $A \supseteq B \supseteq C$ is a unital associative ring $A$ with subring $B$, and $C$ a subring of $B$, so that $1_C = 1_B = 1_A$, which is denoted by $1$.  Sugano in \cite{S} defines $B$ to be a \textit{left relative separable extension of $C$ in $A$} (or briefly,
a left relative separable tower) 
if the $B$-$A$-epimorphism $\mu: B \otimes_C A \rightarrow A$, defined by
$\mu(b \otimes a) = ba$, is split; equivalently, there is a $B$-central element
$e \in (B \otimes_C A)^B$ such that $e^1 e^2 = 1$ (where $e = e^1 \otimes e^2$ is modified Sweedler notation suppressing a  finite sum of simple tensors).
Similarly one defines $B$ to be a right relative separable extension of $C$ in $A$
by requiring $\mu: A \otimes_C B \rightarrow A$ to be a split $A$-$B$-epimorphism. 
The next lemma notes that a separable extension $B \supseteq C$ always give rise to a left and right relative separable extension of $C$ in any over-ring $A$. 

\begin{lemma}
\label{lemma-relsep}
Let $A \supseteq B \supseteq C$ be a tower of rings.  If $B \supseteq C$ is a separable extension, then $B$ is a left and right relative separable extension of $C$ in $A$.  Conversely, if $A \supseteq B$ is a split extension and $B$ is a left or right relative separable extension of $C$ in $A$, then $B \supseteq C$ is a separable extension. 
\end{lemma}
\begin{proof}
Let $e \in (B \otimes_C B)^B$ satisfy $\mu(e) = e^1 e^2 = 1$, the separability condition on $B \supseteq C$. Then $e \in (B \otimes_C A)^B \cap (A \otimes_C B)^B$ defines mappings $a \mapsto ae$ and $a \mapsto ea$ splitting 
$\mu_r: A \otimes_C B \rightarrow A$ and $\mu_{\ell}: B \otimes_C A \rightarrow A$, respectively.   

Suppose $E: {}_BA_B \rightarrow {}_BB_B$ is a bimodule projection (equivalently, 
$E(1) = 1$), and $\mu: B \otimes_C A \rightarrow A$ is a split $B$-$A$-epimorphism (by $\sigma: A \rightarrow B \otimes_C A$).  Then $e = \sigma(1)$
is in $(B \otimes_C A)^B$ satisfying $e^1 e^2 = 1$.  Note then that $e^1 \otimes_C E(e^2)$ is a separability element for $B \supseteq C$.  A similar
argument for a right relative separable tower shows that $B \supseteq C$ is separable. 
\end{proof}

For example, suppose $G > H > J$ is a tower of groups (i.e., $H$ and $J$ are subgroups of $G$ where $J \subseteq H$).  Since group algebra extensions are always split,  the lemma implies that a tower of group algebras $KG \supseteq KH \supseteq KJ$ over a commutative ring $K$ is left or right relative separable if and only if $KH \supseteq KJ$ is separable if and only if $|H : J|1$ is an invertible element in $K$.   

The lemma below provides nontrivial examples of relative separable towers; its proof is easy and therefore omitted.  

\begin{lemma}
Suppose $A \supseteq C$ is a separable extension with separability element
$e \in B \otimes_C A$ for some intermediate subring $B$ of $A$ containing $C$.
Then $B$ is a left relative separable extension of $C$ in $A$.
\end{lemma}

\begin{example}
\label{example-matrixalgebra}
\begin{rm}
Let $K$ be a commutative ring and $A = M_n(K)$ the full $K$-algebra of $n \times n$ matrices.  Let $e_{ij}$ denote the matrix units ($i,j = 1,\ldots, n$). Any of the
$n$ elements $e_j = \sum_{i=1}^n e_{ij} \otimes_K e_{ji}$ are separability idempotents for $A$. 
This is also an example of a (symmetric) Frobenius algebra with trace map $T: A \rightarrow K$ having dual bases $e_{ij}, e_{ji}$.  

Let $B_1, B_2$ denote the upper and lower triangular matrix algebras respectively (both of rank $n(n+1)/2$).  Then $B_1$ and $B_2$ are right, respectively left, relative separable algebras in $M_n(K)$, since $e_1 \in $ \newline  $A \otimes_K B_1 \cap 
B_2 \otimes_K A$.  
\end{rm}
\end{example}

\subsection{Left relative H-separable ring towers} In the same paper \cite{S}, Sugano considers a related condition on a tower of rings $A\supseteq B \supseteq C$, given by the condition on $B$-$A$-bimodules,
\begin{equation}
\label{eq: rel-H-sep}
{}_B B \otimes_C A_A \oplus * \cong {}_BA^n_A,
\end{equation}
for some $n \in \N$, 
i.e., $B \otimes_C A$ is isomorphic to a direct summand of $A \oplus \cdots \oplus A$ as natural $B$-$A$-bimodules.   We define  $B$ to be a \textit{left
relative H-separable extension of $C$ in $A$} (or briefly refer to a \textit{left relative H-separable tower}) if it satisfies the condition in (\ref{eq: rel-H-sep}).  Note that if $A = B$ the condition in (\ref{eq: rel-H-sep}) is that of H-separability of $B \supseteq C$; if $B = C$,  the condition becomes trivially satisfied
by any ring extension $A \supseteq B$. 
We note a lemma similar to the one above.

\begin{lemma}
\label{lemma-hsep}
Let $A \supseteq B \supseteq C$ be a tower of rings.  If $B \supseteq C$ is an H-separable extension, then $B$ is a left (and right) relative H-separable extension of $C$ in $A$. 
\end{lemma}
\begin{proof}
Given the H-separability condition on the natural $B$-$B$-bimodules, $B \otimes_C B \oplus * \cong B^m$, we tensor this from the right by the additive functor
$- \otimes_B A_A$.  After a cancellation of the type $B \otimes_B A \cong A$, the
condition in (\ref{eq: rel-H-sep}) results.  Tensoring similarly by  additive functor $A \otimes_B -$ from the category of $B$-$B$-bimodules into the category of $A$-$B$-bimodules results in the obvious right relative H-separable extension condition (${}_AA \otimes_C B_B \oplus * \cong {}_AA_B^n$) on $B$ over $C$ in $A$.  
\end{proof}
Note that if $B$ is an H-separable extension of $C$, then $B$ is left and right
relative separable and relative H-separable extension of $C$ in any over-ring $A$,
since H-separable extensions are separable extensions \cite{NEFE} (and applying both lemmas).

\subsection{Preliminaries on subring normality and depth} Let $A$ be a unital associative ring.  The category of  right modules over $A$ will be denoted
by $\M_A$.  
Two modules $M_A$ and $N_A$ are \textit{H-equivalent} (or similar) if $M \oplus * \cong  N^q $ and $N \oplus * \cong  M^r$
for some $r,q \in \N$ (sometimes briefly denoted by  $M \sim N$). It is well-known that H-equivalent modules have Morita equivalent endomorphism rings.  

 Let $B$ be a subring of $A$ (always supposing $1_B = 1_A$).  Consider the natural bimodules ${}_AA_A$, ${}_BA_A$, ${}_AA_B$ and
${}_BA_B$ where the last is a restriction of the preceding, and so forth.  Denote the tensor powers
of ${}_BA_B$ by $A^{\otimes_B n} = A \otimes_B \cdots \otimes_B A$ for $n = 1,2,\ldots$, which is also a natural bimodule over  $B$ and $A$ in any one of four ways;     set $A^{\otimes_B 0} = B$ which is only a natural $B$-$B$-bimodule.  

\begin{definition}
\label{def-depth}
If $A^{\otimes_B (n+1)}$ is H-equivalent to $A^{\otimes_B n}$ as $X$-$Y$-bimodules,  one says $B \subseteq A$ (or $A \supseteq B$) has  
\begin{itemize}
\item depth $2n+1$ if $X = B = Y$;
\item left depth $2n$ if $X = B$ and $Y = A$;
\item right depth $2n$ if $X = A$ and $Y = B$;
\item H-depth $2n-1$ if $X = A = Y$.
\end{itemize}
(Valid for even depth and H-depth if $n \geq 1$ and for odd depth if $n \geq 0$.) 
Note that $B \subseteq A$ having depth $n$ implies it has depth $n +1$. Similarly
if $B \subseteq A$ has H-depth $2n-1$, then it has H-depth $2n+1$
(and depth $2n$).  
Define
minimum depth $d(B,A)$, and minimum H-depth $d_H(B,A)$ to be the least depth, or H-depth, satisfied by $B \subseteq A$; if $B \subseteq A$ does not have finite depth,
equivalently finite H-depth, set $d(B,A) = d_H(B,A) = \infty$. 
\end{definition}

 For example, $B \subseteq A$ has depth $1$ iff ${}_BA_B$ and ${}_BB_B$ are H-equivalent \cite{BK2}.
Equivalently, ${}_BA_B \oplus * \cong {}_BB_B^n$ for some $n \in \N$ \cite{LK2012}.  This in turn is
equivalent to there being $f_i \in \Hom ({}_BA_B, {}_BB_B)$ and $r_i \in A^B$ such that
$\id_A = \sum_i f_i(-) r_i$, the classical central projectivity condition \cite{M}.  
  In this case, it is easy to show that $A$ is ring isomorphic to $B \otimes_{Z(B)} A^B$ where
$Z(B), A^B$ denote the center of $B$ and centralizer of $B$ in $A$. From this we deduce immediately 
that a centrally projective ring extension $A \supseteq B$ (equivalently, depth $1$ extension) has centers
satisfying  $Z(B) \subseteq Z(A)$, a condition of Burciu that characterizes depth~$1$ for
a semisimple complex subalgebra pair $B \subseteq A$.  Depth $1$ subgroups are normal
with one other condition on centralizers that depends on the commutative ground ring \cite{BK}.  

For another and important example of depth, the subring $B \subset A$ has right depth $2$ iff ${}_AA_B$
and ${}_A A \otimes_B A_B$ are similar; equivalently, 
\begin{equation}
\label{eq: rD2}
{}_AA \otimes_B A_B \oplus * \cong {}_AA_B^n
\end{equation}
 for some $n \in \N$.  If $A = K G$ is a group algebra of a  finite group $G$, over
a commutative ring $K$, and $B = K H$ is the group algebra of a subgroup $H< G$, then $B \subseteq A$ has right depth $2$ iff $H$ is a normal subgroup of $G$ iff $B \subseteq A$ has left depth $2$ \cite{BDK}; a similar statement is true for a Hopf subalgebra $R \subseteq H$ of finite index and over any field \cite{BDK}.  For this and further reasons mentioned in the first paragraphs of this section we propose the following terminology that is consistent with the literature on normality of subobjects \cite{BDK, BK2, BKK} and of  Galois extensions \cite{KS, LK2008}.  

\begin{definition}
\label{def-normal}
Suppose that $B \subseteq A$ is a subring pair.  We say that $B$ is a right (or left) {\em normal subring} of $A$ if $B \subseteq A$ satisfies the right (or left) depth $2$ condition above.  Similarly, if $B \rightarrow A$ is a ring homomorphism,  we say
that the ring extension $A \| B$ is a right (or left) normal extension if the 
bimodules induced by $B \rightarrow A$ satisfy the right (or left) depth $2$ condition.  A normal extension or normal subring is both left and right normal.  
\end{definition}

For example, centrally projective, or depth $1$, ring extensions are normal extensions.  
As a corollary of \cite[Theorem 3.2]{LK2012} we know that a QF extension is left normal if and only if it is right normal (extending the equivalence of left and right normality for Frobenius extensions in \cite{KS}).  The Galois theory of a normal
extension $A \supseteq B$ with the additional condition that $A_B$ is a balanced module
(with respect to its endomorphism ring $\End A_B$) is briefly summarized as follows:  the
ring $T := (A \otimes_B A)^B \cong \End {}_AA \otimes_B A_A$ has right bialgebroid
structure (and left projective) over the centralizer subring $A^B := R$ \cite{KS} with coaction on $A$, denoted
by $a \mapsto a\0 \otimes_R a\1 \in A \otimes_R T$, having
coinvariant subring $B$, such that \begin{equation}
\label{eq: Galois}
A \otimes_B A \stackrel{\cong}{\longrightarrow} A \otimes_R T
\end{equation}
given by the canonical Galois mapping $a \otimes_B c \mapsto ac\0 \otimes_R c\1$
with inverse given by $a \otimes_R t \mapsto at^1 \otimes_B t^2$ \cite{LK2008}.

\subsection{Sugano's theorem}  Compiling  results in \cite{S} into  a  theorem and using the terminology of depth, we provide a different proof (except in (7) below). 
Let $A \supseteq B \supseteq C$ be a tower of rings, and consider the centralizers
$D := A^C \supseteq A^B := R$.

\begin{theorem} 
\label{th-Sug} 
Suppose $B$ is a left relative H-separable extension of $C$ in $A$; i.e., ${}_BB \otimes_C A_A \oplus * \cong {}_BA_A^n$.  Then the following hold:
\begin{enumerate}
\item $D$ is a left finitely generated projective module over its subring $R$;
\item  as natural $B$-$A$-bimodules, $B \otimes_C A \cong \Hom ({}_RD, {}_RA)$ via \newline $b \otimes a \longmapsto (d \mapsto bda)$;
\item if $B$ is a split extension of $C$, then $D$ is a separable extension of $R$;
\item if $d(B,C) = 1$, then $d_H(R,D) = 1$;
\item if $B \supseteq C$ is a separable extension, then $D \supseteq R$ is a split extension;
\item if  $d_H(C,B) = 1$, then $d(R,D) = 1$;
\item if $A^R = B$ (i.e. $B$ has the double centralizer property in $A$) and $B_C$
is a finitely generated projective module, then the isomorphism given in (2)
restricts to $B \otimes_C B \cong \Hom ({}_RD_R, {}_RA_R)$ and (a)
$d_H(C,B) = 1$ iff $d(R,D) =1$; (b) $B \supseteq C$ is separable iff $D \supseteq R$ is split.
\end{enumerate}
\end{theorem}
{\em Proof}. 
The relative H-separability condition (\ref{eq: rel-H-sep}) condition on the tower
$A \supseteq B \supseteq C$ is clearly equivalent to the two conditions
\begin{itemize}
\item there are $g_1,\ldots,g_n \in \Hom ({}_BB \otimes_C A_A, {}_BA_A) \cong D$ via $g_i \mapsto g_i(1\otimes_C 1)$, and $f_1,\ldots,f_n \in \Hom ({}_BA_A, {}_BB \otimes_C A_A) \cong (B \otimes_C A)^B$ via $f_i \mapsto f_i(1)$
such that $\sum_{i=1}^n f_i \circ g_i = \id_{B \otimes_C A}$;
\item there are $e_i \in (B \otimes_C A)^B$ and $d_i \in D$ for $i = 1,\ldots,n$ such that $1 \otimes_C 1 = \sum_{i=1}^n e_i d_i$,
\end{itemize}
since we define $e_i = f_i(1)$ and $d_i = g_i(1 \otimes_C 1) \in A^C$.  
We will make use of the equation 
\begin{equation}
\label{eq: coord}
1 \otimes_C 1 = \sum_i e_i^1 \otimes e_i^2 d_i
\end{equation}
 below in almost every step of the proof below.
\begin{enumerate}
\item Given $d \in D$, $d = \sum_i e_i^1 d e_i^2 d_i$.  Define $h_i \in \Hom ({}_RD, {}_RR)$ by $h_i(d) = e_i^1 d e_i^2$, thus $d = \sum_i h_i(d) d_i$
is a finite projective bases equation.
\item An inverse to $b\otimes_C a \mapsto \lambda_b \circ \rho_a$ is given by sending $f \in \Hom ({}_RD, {}_R A)$ into $\sum_i e_i f(d_i)$. 
\item Given a bimodule projection $E: B \rightarrow C$, note that
applying $E$ to Eq.~(\ref{eq: coord}) yields $1 = \sum_i E(e_i^1)e_i^2d_i$.  
At the same time, a computation shows that $\sum_i E(e_i^1)e_i^2 \otimes_R d_i
\in (D \otimes_R D)^D$.  
\item Trivially $D \cong \Hom ({}_CC_C, {}_CA_C)$, while $D \otimes_R D \cong
\Hom ({}_CB_C, {}_CA_C)$ via $d \otimes_R d' \mapsto \lambda_d \circ \rho_{d'}$ (with inverse given by $g \mapsto \sum_i g(e_i^1)e_i^2 \otimes_R d_i$
for each $g \in \Hom ({}_CB_C, {}_CA_C)$).   The mapping $\mu: D \otimes_R D \rightarrow D$
corresponds under these isomorphisms to restriction $r: \Hom ({}_CB_C, {}_CA_C) \rightarrow \Hom ({}_CC_C, {}_CA_C)$.  If we have the depth one condition ${}_CB_C \oplus * \cong {}_CC^m_C$, then after applying the additive functor $\Hom (- , {}_CA_C)$
and the ($D$-bimodule) isomorphisms just considered, we obtain $D \otimes_R D \oplus * \cong D^m$ as $D$-$D$-bimodules, the H-depth one condition.  
\item If $e \in (B \otimes_C B)^B$ satisfies $e^1 e^2 = 1$, then the mapping
in (2) applied to $e$ is a bimodule projection in $\Hom ({}_RD_R, {}_RR_R)$.
\item If $B$ is an H-separable extension of $C$, there are $t$ elements $z_i \in (B \otimes_C B)^B$ and $t$ elements $r_i \in B^C$ such that $1 \otimes_C 1 = \sum_{i=1}^t z_i r_i$ \cite{NEFE}.  But $ B^C \subseteq A^C = D$
and $d \mapsto z_i^1 d z_i^2$ defines $t$ mappings in $h_i \in \Hom ({}_RD_R, {}_RR_R)$ such that $d = \sum_i h_i(d) r_i$, a centrally projective bases equation for $D \supseteq R$, thus $d(R,D) = 1$.  
\item
First note that ${}_AA \otimes_RD_D \cong {}_A \Hom (B_C, A_C)_D$ via
$a \otimes d \mapsto \lambda_a \circ \rho_d$
(with inverse $f \mapsto \sum_i f(e_i^1)e_i^2 \otimes_R d_i$).  The isomorphism in (2) restricts to the composite isomorphism (using Proposition 20.11 in \cite{AF}) of
$$ {}_B B \otimes_C B_B \cong {}_B B \otimes_C \Hom ({}_AA_R, {}_AA_R)_B \cong {}_B\Hom({}_A\Hom(B_C, A_C)_R, {}_AA_R) $$
(since $A^R = B (\cong \Hom ({}_AA, {}_AA)^R)$ and $B_C$ is finite projective)
$$\cong {}_B\Hom({}_AA \otimes_R D_R, {}_AA_R)_B \cong \Hom({}_RD_R, {}_R\Hom ({}_AA, {}_AA)_R) $$
$$ \cong {}_B\Hom ({}_RD_R, {}_RA_R)_B. $$  

\item (7a $\Leftarrow$) Suppose $h_j \in \Hom ({}_RD_R, {}_RR_R)$ and $w_j \in D^R$ satisfy $\id_D = \sum_j h_j(-) w_j$.  Then using the isomorphism in (7)
there are $e_j \in (B \otimes_C B)^B \cong \Hom ({}_RD_R, {}_RR_R)$
such that $h_j (d) = e_j^1de_j^2$ for all $d \in D$.  Note that $w_j \in D^R \subseteq A^R = B$ and $w_j \in D = A^C$, whence $w_j \in B^C$.  
It follows from the isomorphism (7) that $1 \otimes_C 1 = \sum_j e_j w_j$
an equivalent condition for H-separability,  $d_H(C,B) =1$.

\item (7b $\Leftarrow$)  Given a projection $E: {}_RD_R \rightarrow {}_RR_R$ 
one notes that $E \in \Hom ({}_RD_R, {}_RA_R)^B \cong (B \otimes_C B)^B$,
so that there is $e \in (B \otimes_C B)^B$ such that $E(d) = e^1de^2$ for all
$d \in D$.  In particular, $e^1e^2 = E(1) = 1$.  
\end{enumerate}   

Since H-separability implies separability for ring extension, we might expect
some mild condition should imply the same for towers of rings.  The next corollary addresses this question.
\begin{cor}
Suppose $A \supseteq B \supseteq C$ is a left relative H-separable tower satisfying 
$D = A^C$ is a left split extension of $R = A^B$.  Then $A \supseteq B \supseteq C$ is left relative separable. 
\end{cor}
\begin{proof}
Applying (2) of Sugano's theorem, note that $\mu: B \otimes_C A \rightarrow A$
corresponds to the $A$-dual of the inclusion $\iota: {}_RR \rightarrow {}_RA$, which is $\iota^*: \Hom ({}_RD,{}_RA) \rightarrow \Hom ({}_RR,{}_RA) \cong {}_BA_A$.  If $\iota$ is a split monic, then $\iota^*$ and $\mu$ are split $B$-$A$-epimorphisms. 
\end{proof}
\begin{cor}
\label{cor-interesting}
Suppose $K$ is a commutative ring, $A$ is a $K$-algebra with $B$ a $K$-subalgebra satisfying the $B$-$A$-bimodule generator condition $B \otimes_K A \oplus * \cong A^n$ (for some $n \in \N$).
Let $R$ be the centralizer $ A^B$.  The following holds: 
\begin{enumerate}
\item ${}_RA$ is a finite projective module;
\item $B \otimes_K A \cong \End {}_RA$ via $b \otimes a \mapsto \lambda_b \circ \rho_a$;
\item if $B$ has a $K$-linear projection onto $K1$, then $A$ is a separable extension of $R$;
\item if $B$ is finite projective as a $K$-module, then $A$ is an H-separable extension of $R$;
\item if $B$ is a separable $K$-algebra, then $A$ is a progenerator $B$-$A$-bimodule and  $A \supseteq R$ is a split extension;
\item if $B$ is an Azumaya algebra with center $Z$ such that $Z \otimes_K  Z\cong Z$ (via $\mu$),
then $A$ is centrally projective over its subalgebra $R$;
\item if $A^R = B$ and $B$ is a finite projective $K$-module, then the isomorphism in (2) restricts to $B \otimes_K B \cong \End {}_RA_R$
and (5), (6) become iff statements.
\end{enumerate}
\end{cor}
\begin{proof}
The proof follows from Sugano's theorem by letting $C = K1$, the unit subalgebra in $A$ and $B$.
In   (6) and (7) we make use of Sugano's characterization of a H-separable $K$-algebra $B$ as being
Azumaya over its center $Z$ subject to the condition $\mu: Z \otimes_K Z \stackrel{\cong}{\longrightarrow} Z$.  In (5), the bimodule ${}_BA_A$ is already noted to be a generator, and it is
finite projective, since given any $B$-$A$-bimodule ${}_BM_A$ and $B$-$A$-epimorphism $\phi:  M \rightarrow A$, $\phi$ is split by $a \mapsto e^1me^2 a$ where $e \in B^e$ is a separability idempotent and $\phi(m) = 1$. 
\end{proof}

A converse to Lemma~\ref{lemma-hsep} is given in the following.  The hypothesis of cleft extension in  the corollary
is fullfilled for example by any finite-dimensional $A$ with nilradical $J$ and separable subalgebra
$B \cong A/J$ (using Wedderburn's Principal Theorem).  
\begin{cor}
Suppose $\pi: {}_BA_B \rightarrow {}_BB_B$ is a ring epimorphism splitting $A \supseteq B$ (a so-called cleft extension), and $C$ is a subring of $B$ such that the left relative H-separable tower condition holds.
Then $B \supseteq C$ is H-separable (i.e., $d_H(C,B) = 1$).  
\end{cor}
\begin{proof}
Apply $\id_B \otimes_C \pi$ to the decomposition of $1 \otimes_C 1$ given in Eq.~(\ref{eq: coord}). We obtain $1 \otimes_C 1 = \sum_i e_i^1 \otimes \pi(e_i^2) \pi(d_i)$ where each
$e_i^1 \otimes_C \pi(e_i^2) \in (B \otimes_C B)^B$ and each $\pi(d_i) \in B^C$: possessing Casimir elements and centralizer elements like these characterizes H-separability of $B$ over $C$.  
\end{proof}
\section{Subring depth in a relative separable tower}
\label{two}
The progenerator condition in Corollary~\ref{cor-interesting} is used again in the hypothesis 
of the proposition below.  
\begin{prop}
Suppose a finitely generated projective $K$-algebra $A$ has  subalgebra $B$ such that $A$ is a progenerator $B$-$A$-bimodule.  Then $A \supseteq B$ is left normal. 
\end{prop}
\begin{proof}
Since $\mu: B \otimes_K A \rightarrow A$ splits, ${}_BA_A \oplus * \cong B \otimes_K A$; thus tensoring by $A \otimes_B -$ we obtain
$A \otimes_B A \oplus * \cong A \otimes_K A$ as natural $A$-bimodules.  Since $B$ is a separable algebra, any $B$-module is $K$-relative projective, whence by the hypothesis on $A$, ${}_BA \oplus * \cong {}_BB^m$ and so
$A \otimes_B A \oplus * \cong B \otimes_K A^m$ as $B$-$A$-bimodules.  Since ${}_BA_A$ is a generator, it follows that $B \otimes_K A \oplus * \cong {}_BA_A^q$, whence ${}_B A \otimes_B A_A \oplus * \cong
{}_BA_A^{mq}$, the left depth $2$ condition on $A \supseteq B$.  
\end{proof}
\begin{prop}
\label{prop-equalityofHdepth}
Suppose a tower of rings $A \supseteq B \supseteq C$ satisfies the left relative H-separability condition
${}_BB \otimes_C A_A \oplus * \cong {}_BA^n_A$ and the left relative separability condition
${}_BA_A \oplus * \cong {}_BB \otimes_C A_A$.  Then $d_H(B, A) = d_H(C, A)$.
\end{prop}
\begin{proof}
Tensoring by ${}_AA \otimes_B -$ the left relative H-separability condition yields
${}_AA \otimes_C A_A \oplus * \cong {}_AA \otimes_B A^n_A$.  Tensoring by
${}_AA \otimes_B -$ the left relative separability condition above yields
$A \otimes_B A \oplus * \cong A \otimes_C A$ as natural $A$-bimodules, whence
$A \otimes_B A$ and $A \otimes_C A$ are H-equivalent as $A$-$A$-bimodules.  

 Suppose that $A^{\otimes_B n}$ is H-equivalent to $A^{\otimes_C n}$ for any $m > n \geq 2$. Then $A^{\otimes_B (m-1)}$ and $A^{\otimes_C (m-1)}$ are H-equivalent, so $A \otimes_C A^{\otimes_B (m-1)}$ and $A^{\otimes_C m}$
are H-equivalent, as are $(A \otimes_C A) \otimes_B \cdots \otimes_B A$
and $(A \otimes_B A) \otimes_B \cdots \otimes_B A$.  It follows from this inductive argument 
that $A^{\otimes_B m}$ and $A^{\otimes_C m}$ are H-equivalent as $A$-bimodules for any $m > 1$.  

Suppose $A \supseteq B$ has H-depth $1$, equivalently, $A$ and $A \otimes_B A$ are H-equivalent, which is equivalent to $A$ and $A \otimes_C A$ being H-equivalent iff $A \supseteq C$ has H-depth $1$. From  the definition of H-depth in Section~1 and the H-equivalences noted above, $A \supseteq B$ has H-depth $n$ iff $A \supseteq C$ has H-depth $n$
for any $n \geq 1$.  
\end{proof}
We improve on \cite[Theorem 2.3]{LK2008a} next.
\begin{prop}
Suppose $B$ is an Azumaya $K$-algebra and subalgebra of a finitely generated projective $K$-algebra
$A$.  Then $A \supseteq B$ has depth~$1$.
\end{prop}
\begin{proof}
Since $B$ is Azumaya, it is well-known that $B$ is a progenerator $B^e$-module (e.g.\ \cite{NEFE}).     Since $B$ is a separable $K$-algebra, $B^e$ is a semisimple extension of $K1$.  Then ${}_BA_B$ is $K$-relative projective, therefore ${}_BA_B$ is finite 
projective since $A$ is projective over $K$.  Thus $A \oplus * \cong B\otimes_K B^m$ for some
$m \in \N$.  But $B^e \oplus * \cong {}_BB_B^n$ for some $n \in \N$ since ${}_BB_B$ is a generator.
Putting these together, ${}_BA_B \oplus * \cong {}_BB_B^{mn}$.  
\end{proof}

\subsection{Higman-Jans-like theorem}  Higman's theorem in \cite{H} states that a finite-dimensional group algebra $kG$ where $k$ is a field of characteristic $p$ has finite representation type if and only if the Sylow $p$-subgroup of $G$ is cyclic.  The proof was teased apart by Jans in \cite{J} into two statements about the property of finite representation type of a subalgebra pair of Artin algebras going up or down according to whether
$A$ is a split or separable extension of $B$; e.g., a separable  and finitely generated extension $A \supseteq B$ of Artin algebras where $B$ has finitely many isoclasses of indecomposable modules implies that also $A$
has finite representation type; see also \cite[pp.\ 173-174]{P}.  In generalizing this theorem, we
first need a lemma characterizing left relative separable towers of rings $A \supseteq B \supseteq C$
in terms of modules.
\begin{lemma}
$B$ is a left relative separable extension of $C$ in $A$ if and only if for each module ${}_AM$,
the mapping $\mu_M: B \otimes_C M \rightarrow M$ given by $b \otimes_C m \mapsto bm$ splits
naturally as a left $B$-module epimorphism.  
\end{lemma}
\begin{proof}
($\Rightarrow$) This is clear from tensoring the split epi $\mu: B \otimes_C A \rightarrow A$
by $-\otimes_A M$ to obtain the split epi $\mu_M$.  ($\Leftarrow$)  Apply the hypothesis to $M = A$
and use naturality to obtain a split $B$-$A$-bimodule epi $\mu: B \otimes_C A \rightarrow A$.   
\end{proof}
Let $A$ be an Artin algebra,  $A\!-\!\mbox{\rm mod}$ denote the category of finitely generated left $A$-modules, and $\mbox{add}\, M$ denote the category of summands of finite sums of copies of a module $M$. 
\begin{theorem}
Suppose $A \supseteq B \supseteq C$ is a left relative separable tower of Artin algebras, where ${}_BA$ and ${}_CB$ are finitely generated.  Suppose $C\!-\!\mbox{\rm mod}$
has finitely many isoclasses of indecomposable representatives $V_1,\ldots,V_n$. Then the restriction functor $\Res^A_B: A\!-\!\mbox{\rm mod} \rightarrow B\!-\!\mbox{\rm mod}$ factors through the subcategory $\mbox{\em add}\, \oplus_{i=1}^n  B \otimes_C V_i$. 
\end{theorem}
\begin{proof}
Given $M \in A\!-\!\mbox{\rm mod}$,
its restrictions ${}_BM$ and ${}_CM$ are finitely generated. By the lemma, ${}_BM$ is isomorphic to a direct summand of $B \otimes_C M$.  Since the restriction 
${}_CM \cong \oplus_{i=1}^n n_i V_i$ for some nonnegative integers $n_i$, one obtains
$\mbox{\rm Res}^A_B M \oplus * \cong  \oplus_{i=1}^n n_i B \otimes_C V_i$, which is expressible as
a Krull-Schmidt decomposition into finitely many indecomposable $B$-module summands of
$B \otimes_C V_1, \ldots, B \otimes_C V_n$. 
\end{proof}
Of course if $A = B$ and $B \supseteq C$ is a separable finitely generated extension, then the theorem recovers Jans's, ``$B$ has finite representation type if $C$ has.''   

\subsection{Triviality of  Relatively H-separable Group Algebra Towers} We next note that towers of finite complex group algebras that are left or right relative H-separable extension are just arbitrary group algebra extensions.  

\begin{prop}
\label{prop-char}
Let $A = \C G \supseteq B = \C H \supseteq C = \C J$ where $G > H > J$ is a tower of subgroups of a finite group $G$.  Then  $A \supseteq B \supseteq C$ is a left or right relative H-separable tower of algebras if and only if $H = J$.
\end{prop}
\begin{proof}
Given ${}_BB \otimes_C A_A \oplus * \cong {}_BA_A^n$, we tensor this with simple left $A$-modules, i.e. $G$-modules, and make use of their (irreducible) characters.  Let $\psi \in \mbox{\rm Irr}(G)$ and $\phi \in \mbox{\rm Irr}(H)$.  From the relative H-separable condition above it follows that
$$\bra \Ind^H_J \Res^G_J \psi, \phi \ket_H \leq n \bra \Res^G_H \psi, \phi \ket_H. $$
Letting $\psi = 1_G$, note that $\Res^G_J 1_G= 1_J$ for instance, so that
$$\bra \Ind^H_J 1_J, \phi \ket_H \leq n \bra 1_H, \phi \ket_H.$$
This last inner product is zero if $\phi \neq 1_H$, so that also $\bra \Ind^H_J 1_J, \phi \ket_H = 0$.  If $\phi = 1_H$, then
$\bra \Ind^H_J 1_J, 1_H \ket_H = \bra 1_J, \Res^H_J 1_H \ket_J = 1$ by Frobenius reciprocity.  
From the orthonormal expansion of $\Ind^H_J 1_J$ in terms of $\mbox{\rm Irr}(H)$, it follows that
$\Ind^H_J 1_J = 1_H$.  Comparing degrees, it follows that $|H : J| = 1$, whence $H = J$. 
The proof using the right relative H-separability condition is a similar use of characters of right modules.
The converse is of course trivial.    
\end{proof}
We have seen in Lemma~\ref{lemma-relsep} that a tower $A \supseteq B \supseteq C$ of arbitrary finite group algebras is always
left or right relative separable if $B \supseteq C$ is a separable extension (iff
$| B : C |$ is invertible in the ground ring).  This follows from the fact that group algebra extensions are  split extensions (since given a subgroup $H < G$, the difference set $G - H$ is closed under multiplication by $H$).  

\section{A Characterization of  normality for progenerator ring extensions}
\label{three}
The next proposition provides an alternative characterization of left relative H-separable condition for a tower $A \supseteq B \supseteq C$ where $B_C$ is finitely generated and projective.
\begin{prop}
\label{prop-ref}
Suppose $A \supseteq B \supseteq C$ is a tower of rings such that the natural module $B_C$ is finite
projective.  Then the left relative H-separable condition~(\ref{eq: rel-H-sep}) is equivalent to the condition, ($\exists n \in \N:$) 
\begin{equation}
\label{eq: equiv}
{}_A\Hom (B_C, A_C)_B \oplus * \cong {}_AA^n_B.
\end{equation}
\end{prop}
\begin{proof}
Apply  $\Hom (-_A,A_A)$, an additive functor from the category of $B$-$A$-bimodules into the category
of $A$-$B$-bimodules, to (\ref{eq: rel-H-sep}), ${}_BB \otimes_C A_A \oplus * \cong {}_BA_A^n$. 
Note that $\Hom (B \otimes_C A_A, A_A) \cong \Hom (B_C, A_C)$ as natural $A$-$B$-bimodules via
$F \mapsto F(- \otimes_C 1)$ with inverse $$f \longmapsto (b \otimes_C a \mapsto f(b)a) $$
for every $f \in \Hom (B_C, A_C)$.  Since ${}_A\Hom (A_A, A_A)_B \cong {}_AA_B$, the condition~(\ref{eq: equiv}) follows without the assumption that $B_C$ is finite projective.

Assuming that $B_C$ is finite projective, it follows that $B \otimes_C A_A$ is finite projective and therefore reflexive.  Then $\Hom ({}_A\Hom(B_C,A_C), {}_AA) \cong B \otimes_C A$ as natural
$B$-$A$-bimodules.  It follows reflexively that condition~(\ref{eq: equiv}) implies condition~(\ref{eq: rel-H-sep}). 
\end{proof}
The next theorem provides many interesting classes of examples of relative H-separable towers of rings.  
\begin{theorem}
\label{th-charD2Hsep}
Suppose $B \supseteq C$ is a ring extension and has the natural module $B_C$ a progenerator.  Let $A := \End B_C$
and $B \into A$ given by the left regular representation $b \mapsto \lambda_b$.  Then $B \supseteq C$ is right normal if and only if the
tower $A \supseteq B \supseteq C$ is left relative H-separable.  
\end{theorem}
\begin{proof}
($\Rightarrow$)  This direction of the proof only requires that $B_C$ is finite projective.  
Given the right normality condition, \begin{equation}
\label{eq: RD2}
{}_BB \otimes_C B_C \oplus * \cong {}_BB_C^m
\end{equation}
 for some
$m \in \N$, apply the bimodule ${}_AB_C$ and the additive functor ${}_A\Hom (-_C, B_C)$ to this.  Note 
that the hom-tensor adjoint relation implies that $\Hom (B \otimes_C B_C, B_C) \cong \Hom (B_C, \Hom (B_C, B_C)_C)$ as natural $A$-$B$-bimodules.  This
obtains condition~(\ref{eq: equiv}), equivalent by the proposition to~(\ref{eq: rel-H-sep}). 

($\Leftarrow$) Since we assume $B_C$ is a progenerator, the rings $C$ and $A$ are
Morita equivalent, with  bimodules
${}_AB_C$ and ${}_C\Hom (B_C, C_C)_A$ forming a Morita context. 
In particular, $\Hom (B_C, C_C) \otimes_A B_C \cong C$ as $C$-bimodules
and $B \otimes_C \Hom (B_C, C_C) \cong A$ as $A$-bimodules.

Supposing that condition~(\ref{eq: equiv}) holds on the tower $C \subseteq B \into A$, we substitute ${}_AA_B = {}_A\Hom (B_C, B_C)_B$ in this condition and apply
the hom-tensor adjoint relation with the last Morita isomorphism to obtain:
$$\Hom (B \otimes_C B_C, {}_AB_C)_B \oplus * \cong {}_AB \otimes_C \Hom (B_C, C_C)_B^n. $$
Tensor this from the left by the additive functor ${}_C\Hom (B_C, C_C) \otimes_A -$ and using the (cancellation) isomorphism of Morita pointed out above. \newline We
obtain
${}_C\Hom (B \otimes_C B_C, \Hom (B_C, C_C) \otimes_A B_C)_B \oplus *  \cong $
$$  {}_C\Hom (B \otimes_C B_C, C_C)_B \oplus * \cong  {}_C\Hom (B_C^n, C_C)_B $$
since $B \otimes_C B_C$ is finite projective and one may apply the well-known
natural isomorphism \cite[Proposition 20.10]{AF}.  Now by reflexivity of the projective modules $B_C^n$ and $B \otimes_C B_C$, we apply to this last isomorphism the additive functor $\Hom (- , {}_CC)$ from the category of
$C$-$B$-bimodules into the category of $B$-$C$-bimodules and obtain
the condition~(\ref{eq: RD2}) with $n = m$.  
\end{proof}

\begin{example}
\begin{rm}
Suppose $d(B,A) = 1$, i.e., a ring $A$ is centrally projective over a subring $B$.  Then $A \cong B \otimes_Z C$ where $Z$ is the center of
$B$ and $C = A^B$.  Then $E \cong \End C_Z$ via restriction of endomorphism
to $C$ and $C_Z$ is a progenerator module.  It follows that $\End C_Z$
is an Azumaya $Z$-algebra. The left relative H-separable tower condition ${}_AA \otimes_B E_E \oplus * \cong {}_AE_E^n$
in the theorem reduces to ${}_CC \otimes_Z E_E \oplus * \cong {}_C E_E$, in
which case  Corollary~\ref{cor-interesting} applies to the $Z$-subalgebra pair $C \into \End C_Z$ via $\lambda$.  
\end{rm}
\end{example}

The left and right relative H-separable conditions on a tower $A \supseteq B \supseteq C$ are equivalent if $B \supseteq C$ is a Frobenius extension, i.e., ${}_CB_B \cong {}_C\Hom(B_C, C_C)_B$ as bimodules and $B_C$ has finite projective bases $\{ b_i \}\subset B, \{ \phi_i \} \subset \Hom (B_C,C_C), (i = 1,\ldots,m)$.  

\begin{prop}
If $B \supseteq C$ is a Frobenius extension, then a tower $A \supseteq B \supseteq C$ is left relative H-separable if and only if
it is right relative H-separable.
\end{prop}
\begin{proof}
We make use of the equivalent condition for left relative H-separable tower in Proposition~\ref{prop-ref}.   We note that $\Hom (B_C, A_C) \cong $    $A \otimes_C \Hom (B_C, C_C)$ via $f \mapsto \sum_i f(b_i)
\otimes_C \phi_i$, with inverse given by the ``one-point projections'' mapping $a \otimes_C \psi \mapsto a\psi(-)$. Observe that this mapping is an $A$-$B$-bimodule isomorphism.  It follows from the Frobenius condition ${}_CB_B \cong {}_B\Hom(B_C, C_C)_B$ that the right relative H-separable condition is satisfied by $A \supseteq B \supseteq C$.  
\end{proof}
As a corollary of this proposition and Theorem~\ref{th-charD2Hsep},  we note that the right normality condition for Frobenius extensions with surjective Frobenius homomorphism is equivalent to left normality condition, another proof in this case of \cite{KS}.

\begin{cor}
\label{cor-leftright}
If $B \supseteq C$ is a Frobenius extension, where $B_C$ is a generator, then $B \supseteq C$ is left normal if and only if it is right normal.
\end{cor}
\subsection{A characterization of normality  for Frobenius extensions}
\label{Fro}
In this subsection we characterize normal (twisted) Frobenius extensions
together with their endomorphism rings as being relative H-separable towers.
We find it convenient to change notation to $A \supseteq B$ being the Frobenius
ring extension and $E := \End A_B$ being the top ring in the tower
$B \subseteq A \into E$ where $A \into E$ is given by $a \mapsto \lambda_a$
and $\lambda_a(x) = ax$.  

Suppose $\beta: B \rightarrow B$ is a ring automorphism of $B$.  Denote
a $B$-module $M_B$ as $M_{\beta}$ if twisted by $\beta$ as follows:
$m \cdot b = m \beta(b)$.  
Recall that a $\beta$-Frobenius (ring) extension $A \supseteq B$ is characterized by having a (Frobenius)
homomorphism  $F: {}_BA_B \rightarrow {}_{\beta}B_B$ satisfying $F(b_1ab_2) = 
\beta(b_1)F(a)b_2$ for each $b_1,b_2 \in B$, $a \in A$.  Dual bases
$\{ x_i \}, \{ y_i \}$ in $A$ satisfy $\sum_{i=1}^n x_iF(y_ia) = a$ and $\sum_{i=1}^n
\beta^{-1}(F(ax_i))y_i = a$ for each $a \in A$.  Equivalently, $A_B$ is finite projective and $A \cong {}_{\beta}\Hom (A_B, B_B)$ as  $B$-$A$-bimodules:  see \cite{NEFE} for more details and references.  

For example, if $\beta$ is an inner automorphism, then $A \supseteq B$ is an (ordinary) Frobenius extensions, such
as a group algebra extension of a group $G$ and subgroup $H$
of finite index $n$.  (Suppose  $g_1,\ldots,g_n$ are the right coset representatives of $H$ in $G$, $K$ an arbitrary commutative
ring, then the group algebra $A = KG$ is a Frobenius extension of the group subalgebra $B = KH$
with $F: A \rightarrow B$ the obvious projection defined by $F(\sum_{g \in G} a_g g) = \sum_{h \in H} a_h h$ and dual bases $x_i = g_i^{-1}, y_i = g_i$.) 

\begin{cor}
\label{cor-Frob}
Suppose $A \supseteq B$ is a $\beta$-Frobenius extension with surjective Frobenius homomorphism $F: A \rightarrow B$.  Let $E : = \End A_B$ and embed $A \into E$ via the left regular representation
$\lambda_a(x) = ax$.  Then the tower of rings $B \subseteq A \into E$ is left relative H-separable if and only if $B \subseteq A$ is  right normal.
\end{cor}
\begin{proof}
Since $F: A \rightarrow B$ is assumed surjective, it follows that $A_B$ (and ${}_BA$ by using equivalently $\beta^{-1} \circ F$) is a  generator.  It also follows from the hypothesis of Frobenius extension
that $A_B$ (and ${}_BA$) are finite projective.  Apply Theorem~\ref{th-charD2Hsep}
to conclude that the left relative H-separable tower condition on $B \subseteq A \into E$ is equivalent
to the right normality condition on $B \subseteq A$.
\end{proof}

Recall that a Hopf subalgebra $R$ is normal in a Hopf algebra $H$ if $R$ is stable
under the left and right adjoint actions of $H$ on $R$.  For group algebra extensions this specializes to the usual notion of normal subgroup.  
\begin{cor}
\label{cor-Hopfish}
A Hopf subalgebra
$R$ of a finite-dimensional Hopf algebra $H$ is normal if and only if the tower of algebras 
$R \subseteq H \into \End H_R$ is left relative H-separable.  
\end{cor}
\begin{proof}
This follows from Corollary~\ref{cor-Frob} and  theorems  that $H$ is a $\beta$-Frobenius extension of $R$ (Oberst-Schneider), and another that $H_R$ is free (Nichols-Zoeller).  
Also, as remarked in the introduction, the equivalence of the normality condition for a  Hopf subalgebra $R \subseteq H$ with the depth
$2$ condition on  the ring extension $R \subseteq H$ follows from  \cite{BK2}.  
\end{proof}
\subsection{Galois correspondence proposal} Again let $E$ denote $\End A_B$. 
The condition on the tower $B \subseteq A \into E$ in the next corollary is called the rD3 condition in \cite{LK2008b}.  The depth three condition on $A \supseteq B$ is that ${}_BA \otimes_B A_B \oplus * \cong {}_BA^m_B$ for some $m \in \N$.    Below we apply the same Frobenius coordinate system as above, but we may assume that the twist automorphism $\beta = \id_B$.  
\begin{cor}
Suppose $A \supseteq B$ is a Frobenius extension with surjective Frobenius homomorphism.  Then
$A \supseteq B$ has depth $3$ if and only if  ${}_EE \otimes_A E_B \oplus * \cong {}_EE^m_B$
for some $m \in \N$. 
\end{cor} 
\begin{proof}
The proof is similar to the proof of Theorem~\ref{th-charD2Hsep}, but using the  $E$-$A$-bimodule isomorphism $E \stackrel{\cong}{\longrightarrow} A \otimes_B A$
given by $f \mapsto \sum_i f(x_i) \otimes y_i$, with inverse mapping given by $a \otimes_B a' \mapsto
\lambda_a \circ F \circ \lambda_{a'}$.   The rest of the proof is left to the reader. 
\end{proof}

The two conditions of ``depth three'' and ``depth two'' on a tower go up and down as follows.  The short proof is left to the reader as an exercise using Proposition~\ref{folklore?}.  
\begin{prop}
Suppose $A \supseteq B \supseteq C \supseteq D$ is a tower of unital subrings.
If $A \supseteq B$ has depth $1$ and $A \supseteq B \supseteq C$ is right relative H-separable, then $B \supseteq C \supseteq D$ satisfies the rD3 condition,  ${}_BB \otimes_C B_D \oplus *$   $ \cong {}_BB_D^n$ for some $n \in \N$.  If $B \supseteq D$ has H-depth $1$ and $B \supseteq C \supseteq D$ satisfies he rD3 condition, then $A \supseteq B \supseteq C$ is right relative H-separable. 
\end{prop}
In \cite{S} the left relative separable and H-separable conditions on towers of rings are used by
Sugano for Galois correspondence  in an H-separable extension in terms of centralizers. 
A final thought is to ask if results for Galois correspondence of a normal extension  in \cite{LK2008b}
(in terms of endomorphism rings) may be improved with the use of the tower condition  studied in this paper.  
\subsection{Acknowledgements}
\small{Research for this paper was funded by the European Regional Development Fund through the programme {\tiny COMPETE} 
and by the Portuguese Government through the FCT  under the project 
\tiny{ PE-C/MAT/UI0144/2011.nts}}.

\end{document}